\documentclass[smallextended,nospthms]{svjour3}

\usepackage{amsthm}
\usepackage{amssymb,color,hyperref,mathrsfs,stmaryrd}
\usepackage{amsmath}
\usepackage{mathtools}

\usepackage[usenames,dvipsnames]{xcolor}

\usepackage[curve,matrix,arrow]{xy}

\numberwithin{equation}{section}
\newtheorem{theorem}[equation]{Theorem}
\newtheorem{lemma}[equation]{Lemma}
\newtheorem{mlemma}[equation]{Main Lemma}

\newtheorem{Th}{Theorem}

\theoremstyle{definition}

\newtheorem{rk}[equation]{Remark}

\newtheorem{eg}[equation]{Example}

\newcommand{\bF}{\mathbb{F}}
\newcommand{\bZ}{\mathbb{Z}}
\newcommand{\bQ}{\mathbb{Q}}

\newcommand{\F}{\mathcal{F}}
\newcommand{\G}{\mathcal{G}}

\newcommand{\Hom}{\operatorname{Hom}}
\newcommand{\Aut}{\operatorname{Aut}}
\newcommand{\Inn}{\operatorname{Inn}}
\newcommand{\Rep}{\operatorname{Rep}}
\newcommand{\Nil}{\mathcal{N}\!\mathit{il}}

\newcommand{\ua}{\uparrow}
\newcommand{\m}{\mathcal}
\newcommand{\ov}{\overline}
\newcommand{\im}{\operatorname{im}}

\def \<{\langle }
\def \>{\rangle }

\renewcommand{\phi}{\varphi}
\newcommand{\colim}{\operatorname{colim}}

\newcommand{\prk}{{\operatorname{rk}_p}}
\newcommand{\nsg}{\unlhd} 

\newcommand*{\xlongrightarrow}[1]{\xrightarrow{\hspace{.2cm}{#1}\hspace{.2cm}}}

\def\co{\colon\thinspace}

\title{Group cohomology and control of $p$--fusion}

\author{David~J.~Benson \and Jesper~Grodal \and Ellen~Henke
\thanks{All three authors were supported by the Danish National Research
  Foundation through the Centre for Symmetry and
  Deformation (DNRF92). The second author was also supported by a
  European Science Foundation EURYI grant.}}
\authorrunning{D.\ J.\ Benson, J.\ Grodal, and E.\ Henke}

\institute{
D.\ J.\ Benson
\at 
Institute of Mathematics, University of Aberdeen, Aberdeen AB24 3UE,
United Kingdom
\and
J.\ Grodal and E.\ Henke
\at
Department of Mathematical Sciences,  University of
  Copenhagen, Copenhagen, Denmark
\\\email{jg@math.ku.dk , henke@math.ku.dk}}

\begin{document}

\maketitle

\begin{abstract}
We show that if an inclusion of finite groups $H \leq G$ of index
prime to $p$ induces a homeomorphism of mod $p$ cohomology varieties,
or equivalently an $F$--isomorphism in mod $p$ cohomology, then $H$ controls $p$--fusion
in $G$, if $p$ is odd. This generalizes classical results of Quillen
who proved this when $H$ is a Sylow $p$--subgroup, and furthermore
implies a hitherto difficult result of Mislin about cohomology
isomorphisms. For $p=2$ we give analogous
results, at the cost of replacing mod $p$ cohomology with higher
chromatic cohomology theories.

The results are consequences of a general algebraic
theorem we prove, that says that
isomorphisms between $p$--fusion systems over the same finite $p$--group
are detected on ele\-mentary
abelian $p$--groups if $p$ odd and abelian $2$--groups of exponent at most
$4$ if $p=2$.
\end{abstract}

\section{Introduction}
The variety of the mod $p$ cohomology ring of a finite group was first
studied by
Quillen in his fundamental 1971 paper
 \cite{Quillen:1971b+c}, and has been a central tool in group
 cohomology since then. The variety describes the mod $p$ group cohomology ring
 up to {\em $F$--isomorphism}, i.e., a ring homomorphism
 with nilpotent kernel and where every element in the target
 raised to a $p^k$th power lies in the image; see
 \cite[Prop.~B.8-9]{Quillen:1971b+c} and Remark~\ref{fisorem}.

Quillen's first application of the theory was to show in \cite{Quillen:1971a} that if the
Sylow $p$--subgroup inclusion $S \leq G$ induces an $F$--isomorphism
on mod $p$ cohomology
then $S$ controls
$p$--fusion in $G$, if $p$ is odd, which in this case means that
$G$ is $p$--nilpotent.  
Quillen's result has subsequently been revisited
in a number of contexts \cite{Henn:1990a,Brunetti:1998a,Gonzalez-Sanchez:2010a,Isaacs/Navarro:2010a,arXiv:1107.5158v2}, however all retaining the
hypothesis that $S$ is a Sylow $p$--subgroup in $G$.

The main goal of this paper is to considerably strengthen Quillen's
result by replacing $S$ by an arbitrary subgroup $H$ of $G$ containing
$S$, thereby moving past $p$--nilpotent groups to all finite groups.
 We recall that for $S \leq H \leq
G$,  $H$ is said to {\em control $p$--fusion} in $G$, if pairs of
tuples of elements of $S$ are conjugate in $H$ if they are conjugate
in $G$, or equivalently if for
  all $p$--subgroups $P,Q \leq S$, $N_H(P,Q)/C_H(P)$ equals
  $N_G(P,Q)/C_G(P)$ as homomorphisms from $P$ to $Q$.

\begin{Th}[$F$--isomorphism implies control of $p$--fusion, $p$
  odd]\label{MainCohomThm} Let $\iota\co H \leq G$ be an inclusion of finite groups of
  index prime to $p$, $p$ an odd prime, and consider the induced map
  on mod $p$ group cohomology $\iota^*\co H^*(G;\bF_p) \to H^*(H;\bF_p)$. If for each $x \in
  H^*(H;\bF_p)$, $x^{p^k} \in
  \im(\iota^*)$ for some $k \geq 0$,
then $H$ controls $p$--fusion in $G$.
\end{Th}
Recall that $\iota^*$ is injective by an easy transfer
argument \cite[Prop.~4.2.5]{Evens:1991a}, since $p \nmid
|G:H|$. Hence, the condition above that for each $x\in H^*(H,\bF_p)$
there exists $k\geq 0$ with $x^{p^k}\in\im(\iota^*)$, is in fact
equivalent to $\iota^*$ being an $F$-isomorphism.  Note that by the classical 1956 Cartan--Eilenberg stable elements formula
\cite[XII.10.1]{Cartan/Eilenberg:1956a}, $\iota^*$ is an (actual)
isomorphism if $H$ controls $p$--fusion in $G$, so the converse also holds.

The assumption in Theorem~\ref{MainCohomThm} that $H$ and $G$ share a common
Sylow $p$--subgroup is necessary as the inclusion $C_p \to C_{p^2}$
shows. Likewise the assumption that $p$ is odd is necessary, as
Quillen's original example  $Q_8
< 2A_4 = Q_8 \rtimes C_3$ shows. Stronger yet, we show in
Example~\ref{th:GnHn} that for any $n$
there  exists an inclusion $H \leq G$ of odd index with different
$2$--fusion but  which
induces a mod $2$ cohomology isomorphism modulo
the class of 
$n$--nilpotent unstable modules $\Nil_n$
\cite[Ch.~6]{Schwartz:1994a}; $F$--isomorphism
means isomorphism modulo the largest class $\Nil_1$.

\smallskip

Our proof of Theorem~\ref{MainCohomThm} is purely algebraic: By
\cite[Prop.~10.9(ii)$\Rightarrow$(i)]{Quillen:1971b+c} (or the
algebraic reference \cite{Alperin:2006a}) $F$-isomorphism in mod $p$ group cohomology implies control of fusion on elementary abelian subgroups. Thus, Theorem~\ref{MainCohomThm} follows from the following group
theoretic statement, which is of
independent interest. For $p$ odd it says that if $H$ controls
$p$--fusion in $G$ on elementary abelian $p$--subgroups then it in fact
controls $p$--fusion. We formulate and prove the statement in terms of
fusion systems, and refer the reader for example to
\cite{Aschbacher/Kessar/Oliver:2011a} for definitions and information
about these---we also recap the essential definitions in Section~\ref{grpsec}.

\begin{Th}[Small exponent abelian $p$--subgroups control $p$--fusion] \label{MainThm}
Let $\G \leq \F$ be two saturated fusion systems on the same
finite $p$--group $S$. 
Suppose that
$\Hom_\G(A,B)=\Hom_\F(A,B)$  for all $A,B \leq S$ with $A,B$
elementary abelian if $p$ is odd, and abelian of exponent at
most $4$ if $p=2$. Then $\G=\F$.

\end{Th}
Our proof of this theorem is rather short. In outline we use Alperin's Fusion Theorem to reduce to a situation where we can apply 
results of J.~G.~Thompson on $p'$--automorphisms of $p$--groups
\cite[Ch.~5.3]{Gorenstein:1968a}. Consequently, our proof of Theorem \ref{MainCohomThm} is also relatively elementary. In particular, at odd primes, we obtain a comparatively simple algebraic proof of Mislin's Theorem. This theorem states that, for a homomorphism $\phi\co H \to G$ of finite groups, which induces an isomorphism in mod $p$ group cohomology, $|\ker(\phi)|$ and
$|G:\phi(H)|$ are coprime to $p$ and $\phi(H)$ controls $p$--fusion
in $G$. Here the first part is a 1978 theorem of Jackowski
\cite[Thm.~1.3]{Jackowski:1978a}. (Jackowski gave a topological
argument, but a short algebraic proof exists via Tate
cohomology; see \cite[Thm.~5.16.1]{Benson:1991b} with ${\mathbb Z}$
replaced by ${\mathbb Z}_{(p)}$.) So the proof of Mislin's Theorem reduces quickly to the situation that $\phi$ is an inclusion of finite groups of index prime to $p$, where the statement follows from Theorem~\ref{MainCohomThm} if $p$ is odd. Mislin's original proof  of his theorem uses the
Dwyer--Zabrodsky theorem \cite{Dwyer/Zabrodsky:1987a} in algebraic
topology, whose proof again relies on Lannes' theory \cite{Lannes:1992a}, extending Miller's proof of the
Sullivan conjecture \cite{Miller:1984a}.
In the early 1990s, for example at the 1994 Banff
conference on representation theory, Alperin made the highly publicized challenge
to find a purely algebraic proof of Mislin's theorem, and this was pursued
by many authors. 
Symonds \cite{Symonds:2004a},
following an idea of Robinson \cite[\S 7]{Robinson:1998a}, 
provided an algebraic reduction of the problem to a statement about 
cohomology of trivial source modules, which he then proved
topologically. Algebraic proofs were finally completed
independently by Hida \cite{Hida:2007a} and
Okuyama \cite{Okuyama:2006a}, who gave algebraic proofs of 
Symonds' statement, through
quite delicate arguments in modular representation theory. 
(See also e.g., \cite{Alperin:2006a} and \cite{Symonds:2007b}.) 

\smallskip

We now come to a further application of Theorem~\ref{MainThm}. As remarked above, the assumptions in Theorem~\ref{MainCohomThm} that $p$ is odd and $|G:H|$ is prime to $p$ are both in fact necessary. Switching from mod $p$ cohomology to generalized cohomology theories, we can
however combine the methods of Theorem~\ref{MainThm} with Hopkins--Kuhn--Ravenel (HKR) generalized
character theory \cite{Hopkins/Kuhn/Ravenel:1992a,Hopkins/Kuhn/Ravenel:2000a} to obtain
a statement that holds for all primes, and that also avoids the assumption that $H$
and $G$ share a common Sylow $p$--subgroup.

\begin{Th}[Chromatic group cohomology isomorphism implies control
  of $p$--fusion] \label{KnThm}
Let $\phi\co H \to G$ be a homomorphism of finite groups, and let $E(n)$ denote height $n$ Morava $E$--theory at
a fixed prime $p$. Suppose that $\phi$ induces an isomorphism $$\phi^*\co E(n)^*(BG)[ \frac1p]
\xlongrightarrow{\sim} E(n)^*(BH)[\frac1p]$$
for some $n \geq \prk(G)$. Then
$|\ker(\phi)|$ and $|G: \phi(H)|$ are prime to $p$, and $\phi(H)$ controls $p$--fusion
in $G$.\end{Th}
 In fact our proof works not just for $E(n)$ but for
any height $n$ cohomology theory satisfying the assumptions listed in
\cite[Thm.~C]{Hopkins/Kuhn/Ravenel:2000a}.  We
recall that for height $n$ Morava $E$--theory, $E(n)^*(pt) = W(\bF_{p^n})\llbracket
w_1,\ldots,w_{n-1}\rrbracket [u,u^{-1}]$, with $W(\bF_{p^n})$ the
unramified extension of degree $n$ of the $p$--adic integers, $|w_i| =
0$ and $|u|=-2$.
 As usual, the notation $[\frac1p]$ means that we invert
$p$ after taking cohomology, producing a $\bQ_p$--algebra. 

The converse to Theorem~\ref{KnThm} is clear, e.g., by the standard
Cartan--Eilenberg stable elements formula and the fact that a mod $p$ cohomology isomorphism of spaces induces an $E(n)^*$--isomorphism. Theorem~\ref{KnThm} also provides a new proof
of a strong form of Mislin's
theorem, assuming only isomorphism in large degrees. This proof is valid at all
primes, but replaces the reliance on Quillen's
variety theory by the (currently) less algebraic HKR character
theory; indeed the proof mirrors that of Atiyah's 1961 $p$--nilpotence
criterion \cite{Quillen:1971a,Atiyah:1961a}, replacing
$K$--theory by higher chromatic $E(n)$--theories; see Remark~\ref{rem:highdegcohom2}.

A $p$--rank restriction in Theorem~\ref{KnThm} is indeed necessary:
 $\bF_{p^2} \rtimes \bF_{p^2}^\times < (\bF_{p^2} \rtimes \bF_{p^2}^\times) 
\rtimes \Aut(\bF_{p^2})$ is an example of an inclusion of groups of
index prime to $p$, for $p$ odd, which is an $E(1)^*[\frac1p]$--equivalence, by
HKR character theory \eqref{HKR-theorem}, but with different
$p$--fusion; the same example with $\bF_{p^2}$ replaced by $\bF_{2^3}$
works for
$p=2$. We speculate that the bound $n \geq \prk(G)$ we give may be
close to optimal, but we currently do not know an example to this effect.

Finally, we remark that isomorphism on $E(n)^*$ is equivalent to
isomorphism on $n$th Morava $K$--theory $K(n)^*$, whereas an
$E(n)^*[\frac 1p]$--isomorphism is a priori significantly weaker---see
Remark~\ref{spec-remark} for a variety interpretation of
Theorem~\ref{KnThm} and
Remark~\ref{KnWhitehead} for the connection to other stable homotopy theory results.

\smallskip

To prove Theorem~\ref{KnThm} we need the following variant of 
Theorem~\ref{MainThm}, where we drop the assumption of a common Sylow
$p$-subgroup, but on the other hand assume the same fusion on all
abelian $p$--subgroups---it again appears to be new, even in special
cases. 

\begin{Th}[Abelian $p$--subgroups control fusion]
\label{KnThm-algebraic} Assume that a finite group homomorphism
$\phi\co H \to G$ induces a bijection $$\Rep(A,H) \xlongrightarrow{\sim} \Rep(A,G)$$
for all finite abelian $p$--groups $A$ with $\prk(A) \leq \prk(G)$. Then
$|\ker(\phi)|$ and $|G: \phi(H)|$ are prime to $p$, and $\phi(H)$ controls $p$--fusion
in  $G$.
More generally, suppose  that $\F$ and $\G$ are saturated fusion
systems on finite $p$--groups $S$ and $T$ respectively, 
and that $\phi\co T \to S$ is a fusion preserving homomorphism inducing a bijection $\Rep(A,\G) \xrightarrow{\sim} \Rep(A,\F)$ for
any finite abelian $p$--group $A$ with $\prk(A) \leq \prk(S)$. Then
$\phi$ induces an isomorphism from $T$ to $S$ and $\G$ to $\F$.
\end{Th}
Here $\Rep(A,G)$ denotes the quotient of $\Hom(A,G)$ where we identify
$\phi$ with $c_g \circ \phi$ for all $g \in G$, and likewise $\Rep(A,\F)$ is the quotient of $\Hom(A,S)$,
identifying two morphisms if they differ by a morphism in $\F$; we
spell out what the assumptions of the theorem mean in
Lemma~\ref{bijection}.

Finally, we remark that Theorems~\ref{MainCohomThm} and \ref{KnThm} can be formulated in
  terms of the fusions systems of the groups, and they should hold for
  abstract fusion systems as well. Indeed, as is clear from our proofs, the
  only missing piece is a reference for the Quillen
  stratification and the HKR character theorem in  that context---we will however not pursue this here.

\section{$p^\prime$--automorphisms of $p$--groups and proofs of Theorems~\ref{MainThm} and \ref{KnThm-algebraic}}\label{grpsec}

The goal of this section is to prove Theorems~\ref{MainThm} and \ref{KnThm-algebraic}
by group theoretic methods, combining manipulations with fusion
systems with results of J.~G.~Thompson on $p^\prime$--automorphisms of
$p$--groups, which can by now be found in textbooks.

Thompson's {\em critical subgroup
theorem} \cite[Lem.~2.8.2]{Feit/Thompson:1963a} (see also the textbook
reference \cite[Thm.~5.3.11]{Gorenstein:1968a}) says that for any
finite $p$--group $P$ there exists a
characteristic subgroup $C$ of $P$ such that $C/Z(C)$ is elementary
abelian, $[P,C] \leq Z(C)$, $C_P(C) = Z(C)$, and every nontrivial
$p'$--automorphism of $P$ restricts to a non-trivial
$p^\prime$--automorphism of $C$.  Our main classical group theoretic
tool in this paper is a variant of that theorem, where instead of a
critical subgroup we use a certain characteristic subgroup of $P$ of small exponent and consider its maximal abelian subgroups.

\begin{theorem}[Small exponent abelian subgroups detect $p'$--automorphisms]\label{ThLemma}
 Let $P$ be a finite $p$--group. There exists a characteristic subgroup
 $D$ of $P$, of exponent $p$ if $p$ is odd and exponent $4$ if
 $p=2$, such that $[D,P]\leq Z(D)$, and such that every non-trivial
 $p^\prime$--automorphism of $P$ restricts to a non-trivial automorphism of
 $D$. Furthermore, for any such $D$ and any
 maximal (with respect to inclusion) abelian
 subgroup $A$ of $D$ we have $A  \nsg P$ and $C_{\Aut(P)}(A)$ is a $p$--group.
\end{theorem}

Note that the example of the extra-special group $p^{1+2}_+$ shows that an
abelian characteristic subgroup that detects $p'$--automorphisms need not exist.
\begin{proof}[Proof of Theorem~\ref{ThLemma}]
Taking $D = \Omega_1(C)$, the subgroup generated by elements of
 order $p$ of a critical subgroup $C$, produces such a subgroup $D$ as
 in the theorem, for
 $p$ odd, as proved in  \cite[Thm.~5.3.13]{Gorenstein:1968a}. For
 $p=2$ the claim holds for $D = \Omega_{2}(C)$, the
subgroup of $C$ generated by elements of order at most $2^2$; we establish this fact 
in Lemma~\ref{p=2} below.

For the last part, let $A$ be a maximal abelian subgroup of $D$ with
respect to inclusion.  Since $[A,P]
\leq Z(D) \leq A$ it follows that $A \unlhd P$. Furthermore, if $B \leq
C_{\Aut(P)}(A)$ is a $p^\prime$--group, then $A \times B$ acts on
$P$ and thus on $D$. Since $A$ is maximal abelian it follows that $C_{D}(A) = A$, and in particular $B$ acts trivially
on $C_{D}(A)$; Thompson's $A \times B$--lemma
\cite[Thm.~5.3.4]{Gorenstein:1968a} now says that $[D,B]=1$ and so $B =1$, and we
conclude that $C_{\Aut(P)}(A)$ is a $p$--group as wanted.
\end{proof}

We now provide a proof of the postponed lemma for $p=2$.

\begin{lemma}\label{p=2}
 Let $P$ be a $2$--group such that $P/Z(P)$ is elementary abelian. Then
 for all $x,y \in P$,  $(xy)^4=x^4y^4$ and in particular $\Omega_2(P)$
 is of exponent at most $4$. Furthermore if  $B$ is a $p^\prime$--group
 of automorphisms of $P$ with $[\Omega_2(P),B]=1$, then $B=1$.
\end{lemma}

\begin{proof} 
Note that $(xy)^2 = x^2(x^{-1}yxy^{-1})y^2$ with all three factors in
$Z(P)$, so 
\begin{equation*}
\begin{split}
(xy)^4 = x^4y^4(x^{-1}yxy^{-1})^2 &=
x^4y^4(x^{-1}(x^{-1}yxy^{-1})yxy^{-1})  \\ &=
x^4y^4(x^{-2}yx^2y^{-1}) = x^4y^4.
\end{split}
\end{equation*}
\indent For the last statement about $p^\prime$--automorphisms we follow \cite[Thm.~5.3.10]{Gorenstein:1968a}.
Let $P$ be a minimal counterexample. If $Q$ is a
proper $B$--invariant subgroup of $P$ then $Q/(Q\cap Z(P))$ is
elementary abelian and $Q\cap Z(P)\leq Z(Q)$, so $Q/Z(Q)$ is
elementary abelian. Moreover, $\Omega_2(Q)\leq \Omega_2(P)$ and thus
$[\Omega_2(Q),B]=1$. So, as $P$ is a minimal counterexample,
$[Q,B]=1$. By \cite[Thm.~5.2.4]{Gorenstein:1968a}, $P$ is
non-abelian. So in particular, $Z(P)$ is a proper characteristic
subgroup of $P$ and thus $[Z(P),B]=1$. 
 We now show that $[P,B]\leq \Omega_2(P)$: Suppose $x\in P$ and $b\in
B$, and note that $x^4\in Z(P)$, as $P/Z(P)$ is elementary
abelian, and thus $(x^4)^b=x^4$, as $[Z(P),B]=1$. Hence
$[x,b]^4=(x^{-1}x^b)^4=x^{-4}(x^4)^b = x^{-4}x^{4}=1$ as wanted,
where we also used the first part of the lemma.
By assumption $[\Omega_2(P),B]=1$, so in particular $[[P,B],B] =1$ by
the above, and we conclude that  $[P,B]=1$, by
\cite[Thm.~5.3.6]{Gorenstein:1968a}. 
\end{proof}

In the case where $\F$ is the fusion system of $G = S \rtimes K$, with $p \nmid |K|$, Theorem~\ref{MainThm} follows
directly from Theorem~\ref{ThLemma}, as the action of elements of $K$
on $S$ is detected by small abelian subgroups of $S$, but the proof of the
general statement requires more work, and here fusion systems
enter in a more prominent way. The arguments can be translated into
the special case of ordinary finite groups, but doing so provides
no essential simplifications, and indeed, from our perspective, the
arguments are considerably shorter and more transparent in the setup of fusion systems.

\smallskip
 
Recall that a {\em saturated fusion system $\F$} on a finite
$p$--group $S$  \cite[Def.~1.2]{Broto/Levi/Oliver:2003a}\cite[Prop.~I.2.5]{Aschbacher/Kessar/Oliver:2011a}
is a category whose objects are the subgroups of $S$, and morphisms
are group monomorphims satisfying axioms which mimic those satisfied
by morphisms induced by conjugation in
some ambient group $G$. More precisely, conjugation by elements in $S$ need to be in
the category, every map needs to factor as an isomorphism followed by
an inclusion, and furthermore two non-trivial conditions need to be
satisfied, called the Sylow and extension axiom, which we recall below
together with some terminology. We refer to
\cite{Aschbacher/Kessar/Oliver:2011a} and 
\cite{Broto/Levi/Oliver:2003a} for detailed information, and also
direct the reader to Puig's original work \cite{Puig:2006a}, where
terminology however differs.
A subgroup $Q \leq S$ is called {\em fully $\F$--normalized} if $|N_S(Q)|$ is
maximal among $\F$--conjugates of $Q$, it is called {\em fully
$\F$--centralized} if the corresponding property holds for the
centralizer, and it is called {\em $\F$--centric} if $C_S(Q') = Z(Q')$ for all
$\F$--conjugates $Q'$ of $Q$. The {\em Sylow axiom} says that if $Q$ is
fully
$\F$--normalized then it is fully $\F$--centralized and $\Aut_S(Q)$ is 
a Sylow $p$--subgroup of $\Aut_\F(Q)$. (Here $\Aut_S(Q)$ means the
automorphisms of $Q$ induced by elements in $S$.) The {\em extension
axiom} says that any morphism $\phi\co Q \to S$ with $\phi(Q)$
fully $\F$--centralized extends to
$$N_\phi = \{ g \in N_S(Q) | {}^\phi (c_g|_Q) \in \Aut_S(\phi(Q))\}.$$

The first tool we need is the following variant of the extension axiom.

\begin{lemma}\label{FSEl1} Fix a saturated fusion system $\F$ on $S$ and
  let
  $\phi\co P \to S$ be {\em any} monomorphism (not necessarily in $\F$). For $Q
  \unlhd P$ and $\psi = \phi|_Q$ the following hold.
\begin{enumerate}
\item\label{a}
$N_\psi \geq P$ and  ${}^\psi\! \Aut_P(Q)=\Aut_{\phi(P)}(\phi(Q))$. 
\item\label{b} If $\psi \in \F$ and $\phi(Q)$ is fully
  $\F$--centralized then $\psi$ extends to $\hat{\psi}\in$ \\  $\Hom_\F(P,\phi(P)C_S(\phi(Q)))$.
\end{enumerate}
\end{lemma}

\begin{proof}
For \eqref{a} we calculate, for any $g \in P$ and $x \in \phi(Q)$,
$$({}^\psi c_g)(x) = \psi \circ c_g \circ
\psi^{-1}(x) = \psi(g \psi^{-1}(x) g^{-1}) = \phi(g) x \phi(g^{-1}) =
c_{\phi(g)}(x),$$ from which it is clear that
$N_\psi \geq P$ and $^\psi \Aut_P(Q)=\Aut_{\phi(P)}(\phi(Q))$. 

For \eqref{b} note that the extension axioms imply that
$\psi$ extends to $\hat{\psi}\in \Hom_\F(P,S)$.  And, since
$\Aut_{\phi(P)}(\phi(Q))={^\psi \Aut_P(Q)}=\Aut_{\hat{\psi}(P)}(\phi(Q))$, where the last
equality is by applying \eqref{a} with $\hat\psi$
in place of $\phi$, we conclude that $\hat{\psi}(P)\leq \phi(P)C_S(\varphi(Q))$ as wanted.
\end{proof}

For the purpose of the next proof, recall that a
proper subgroup $H$ of a finite group $G$ is called {\em strongly
$p$-embedded} if $p$ divides the order of $H$ and, for all $g\in
G\backslash H$, $H\cap {^gH}$ has order prime to $p$. Provided $p$ divides $|G|$, one easily shows
that $H$ is strongly $p$--embedded in $G$ if and only if $H$ contains a
Sylow $p$--subgroup $S$ of $G$ such that $N_G(R)\leq H$ for every
$1\neq R\leq S$ (see for example
\cite[Lem.~17.10]{Gorenstein/Lyons/Solomon:1996a} or \cite[Prop.~5.2]{Quillen:1978a}); in particular an
overgroup of a strongly $p$--embedded subgroup is again strongly
$p$--embedded, if it is a proper subgroup. (Groups with strongly
embedded subgroups play a central role in many aspects of local group theory, and in
particular they show up in connection with Alperin's fusion theorem \cite[Thm.~I.3.6]{Aschbacher/Kessar/Oliver:2011a}, though
we shall only indirectly need them in that capacity here.)

We now give the key step in deducing Theorem
\ref{MainThm} from Theorem~\ref{ThLemma}, providing a way to show
that the fusion in $\F$ and $\G$ agree on all subgroups $P$ by
downward induction starting with $S$.

\begin{mlemma}\label{MainLemma} Let $\G \leq \F$ be two saturated fusion
  systems on the same finite $p$--group $S$, and $P \leq S$ an
  $\F$--centric and fully $\F$--normalized subgroup,
with $\Aut_\F(R)=\Aut_{\G}(R)$ for every $P<R\leq N_S(P)$.
Suppose that there exists a subgroup $Q \unlhd P$ with $\Hom_\F(Q,S)=\Hom_\G(Q,S)$. Then $\Aut_\F(P) = \< \Aut_\G(P),C_{\Aut_\F(P)}(Q)\>$.
\end{mlemma}

\begin{proof}To ease the  notation set $G = \Aut_\F(P)$, $H =
  \Aut_\G(P)$, and $\ov G = G/\Inn(P)$, and denote by
  $\ov{U}$ the image in $\ov G$ of any subgroup $U\leq G$. We want to
  show that $G = \< H, C_G(Q)\>$.

\smallskip
\noindent{\em Step 1:} We first assume in addition that 
\begin{equation*}\label{extra}
\tag{$*$}
C_S(\xi(Q)) \leq P \mbox{ for all }\xi\in H
\end{equation*}
and show that $G=HC_G(Q)$. Let $\gamma\in G$ be arbitrary; set
$\psi=(\gamma^{-1})|_{\gamma(Q)}\in \Hom_\F(\gamma(Q),Q)$. Then $\psi\in
\Hom_\G(\gamma(Q),Q)$ by assumption. We claim
that $Q$ is fully centralized in $\G$, and postpone the proof to
Lemma~\ref{FSEl2} below, since it is a general statement. Granted
this, Lemma~\ref{FSEl1}(\ref{b}), applied to $\gamma^{-1}$ and $\G$ in
the roles of $\phi$ and $\F$, implies that we can extend $\psi\co
\gamma(Q) \to Q$ to  $\hat \psi\co P \to PC_S(Q)$ in $\G$, and, as $C_S(Q) \leq P$
by assumption \eqref{extra}, we conclude that $\hat{\psi}\in
H$. 
Since $\gamma\circ\hat{\psi}\in C_{G}(Q)$, we have $\gamma\in
C_{G}(Q)H$, and, as $\gamma$ was arbitrary, this
yields $G=HC_G(Q)$ as required.

\smallskip\noindent
{\em Step 2:} If $P=S$ assumption \eqref{extra} is automatically
satisfied and the lemma follows from Step~1; likewise we are done if $H = G$.
In this step we
show that if $P < S$ and $H < G$ then 
$\bar H$ is strongly $p$--embedded in $\bar G$.
Consider  $P<R\leq
N_S(P)$. For $\varphi\in
N_G(\Aut_R(P))$ it follows from the extension axiom that $\phi$
extends to $\hat{\varphi}\in \Hom_\F(R,S)$, since $P$ is fully
$\F$--normalized and $R \leq N_\phi$, cf.\
Lemma~\ref{FSEl1}. Furthermore, by Lemma~\ref{FSEl1}\eqref{a},
$\Aut_R(P)={^\varphi \Aut_R(P)}=\Aut_{\hat{\varphi}(R)}(P)$, so since
$C_S(P)\leq P$ by $\F$--centricity of $P$, we have
$\hat{\varphi}(R)=R$. It then follows from
our hypothesis that $\hat{\varphi}\in \Aut_{\G}(R)$ and thus
$\varphi\in H$. We conclude that
 $\ov H$ is strongly
$p$--embedded in $\ov G$.

\smallskip
\noindent{\em Step 3:} Finally set $H_0 = \<H,C_{G}(Q)\>$, and suppose for contradiction that there exists $\chi \in G \setminus
H_0$.
Then by Step~2, $\ov H_0$ is a
strongly $p$--embedded subgroup of $\ov G$, so in particular $\Aut_S(P)\cap {}^\chi H_0=\Inn(P)$ as $\Aut_S(P)\leq H_0$. Note that  $ C_{G}(\chi(Q)) = {}^{\chi} C_{G}(Q) \leq
{}^\chi H_0$, so $C_{\Aut_S(P)}(\chi(Q)) \leq \Aut_S(P) \cap {}^{\chi}H_0=\Inn(P)$. Hence $N_S(P)\cap C_S(\chi(Q))
\leq P$, using that $P$ is centric. Now, as $C_S(\chi(Q))P$ is a
$p$--group, $C_S(\chi(Q))\leq P$ (see
\cite[Thm.~2.3.4]{Gorenstein:1968a} for this elementary property of
finite $p$--groups). 
Note that $\xi\circ \chi\in
G\backslash H_0$ for any $\xi \in H$; so as $\chi$ was arbitrary the
argument actually shows that $C_S(\xi(\chi((Q))))\leq P$ for any $\xi
\in H$. But now (\ref{extra}) holds with $\chi(Q)$ in place of $Q$. Observe also that $\Hom_\F(\chi(Q),S)=\Hom_\G(\chi(Q),S)$. For if $\phi\in \Hom_\F(\chi(Q),S)$ then $\phi\circ \chi$ and $\chi^{-1}$ are in $\Hom_\F(Q,S)=\Hom_\G(Q,S)$, so $\phi=(\phi\circ \chi)\circ \chi^{-1}\in \Hom_\G(\chi(Q),S)$. 
Therefore Step~1 gives that $G =HC_G(\chi(Q))$. As $C_G(\chi(Q))$ is conjugate to $C_G(Q)$ in $G$, it follows that $C_G(\chi(Q))$ is conjugate to $C_G(Q)$ by an element of $H$ and thus $G=HC_G(Q)$. This is a contradiction, and we conclude
that $G = H_0$ as wanted.
\end{proof}

We next prove the postponed lemma. 

\begin{lemma} \label{FSEl2}  Let $\F$ be a saturated fusion system on $S$
  and  suppose that $Q \unlhd
  P \leq S$, with $P$ fully $\F$--normalized, and $C_S(\xi(Q)) \leq P$ for all $\xi\in \Aut_\F(P)$. Then $Q$ is fully $\F$--centralized.
\end{lemma}

\begin{proof}
By \cite[Lem.~2.6]{Linckelmann:2007a} we may choose $\alpha\co N_S(Q) \to S$ in $\F$ such that $\alpha(Q)$ is fully
normalized. Furthermore, as $P$ is fully normalized, again by \cite[Lem.~2.6]{Linckelmann:2007a}, there is $\beta\in \Hom_\F(N_S(\alpha(P)),N_S(P))$ such that
$\beta(\alpha(P))=P$.
Then 
$$\beta(C_S(\alpha(Q))\cap N_S(\alpha(P)))\leq C_S(\beta(\alpha(Q)))\leq P=\beta(\alpha(P))$$ 
where the second inclusion follows by assumption as $\beta\circ\alpha$
restricts to an element of $\Aut_\F(P)$. This yields
$C_S(\alpha(Q))\cap N_S(\alpha(P))\leq\alpha(P)$, so $$N_{C_S(\alpha(Q))\alpha(P)}(\alpha(P))=(C_S(\alpha(Q))\cap N_S(\alpha(P)))\alpha(P)=\alpha(P).$$ Thus, as 
$C_S(\alpha(Q))\alpha(P)$ is a $p$--group, it follows from \cite[Thm.~2.3.3(iii) and Thm.~2.3.4]{Gorenstein:1968a} that $C_S(\alpha(Q))\leq
\alpha(P)$. Hence,   $C_S(\alpha(Q))=C_{\alpha(P)}(\alpha(Q))=\alpha(C_P(Q))=\alpha(C_S(Q))$ where the last equality holds since our assumption gives $C_S(Q)\leq P$. It follows $|C_S(\alpha(Q))|=|\alpha(C_S(Q))|=|C_S(Q)|$; so $Q$ is fully $\F$--centralized as $\alpha(Q)$ is fully $\F$--centralized.
\end{proof}

\begin{proof}[Proof of Theorem~\ref{MainThm}] 
By Alperin's fusion theorem, $\F$ is generated by $\F$--auto\-morphisms of fully $\F$--normalized and
  $\F$--centric subgroups;   see
  \cite[Thm.~I.3.6]{Aschbacher/Kessar/Oliver:2011a}  (in fact we only
  need
  ``$\F$--essential'' subgroups and $S$). We want to show that
  $\Aut_\G(P) = \Aut_\F(P)$ for all $P \leq S$; by downward induction
  on the order we can assume that $\Aut_\G(R) = \Aut_\F(R)$ for all subgroups $R\leq S$
  with $|R|>|P|$, and by the fusion theorem we can furthermore assume
  that $P$ is $\F$--centric and fully $\F$--normalized.
Now choose a characteristic subgroup $D$ of $P$ as described in
Theorem~\ref{ThLemma}, and a maximal abelian subgroup $A$ of
$D$, and recall that the theorem tells us that $A \nsg P$ and
that $C_{\Aut_\F(P)}(A)$ is a $p$--group. As $P$ is
fully $\F$--normalized, $\Aut_S(P)$ is a Sylow $p$--subgroup of
$\Aut_\F(P)$, 
 so if we replace $A$ by a conjugate of $A$ under
$\Aut_\F(P)$, we can arrange that $C_{\Aut_\F(P)}(A)\leq \Aut_S(P)\leq
\Aut_\G(P)$. 
But $A$ also satisfies the assumptions on $Q$ in Lemma~\ref{MainLemma},
so $\Aut_\F(P) = \< \Aut_\G(P), C_{\Aut_\F(P)}(A)\>$, and we
conclude that $\Aut_\G(P) = \Aut_\F(P)$ as wanted.
\end{proof}

We now head towards a proof of Theorem~\ref{KnThm-algebraic}. Recall
that for $Q$ a group and $\F$ a fusion system on $S$ we define
$\Rep(Q,\F) = \Hom(Q,S)/\F$ as the quotient of $\Hom(Q,S)$ under $\F$--conjugation,
i.e., where we identify $\phi \in \Hom(Q,S)$ with  $\alpha
\circ \phi $,
for all $\alpha \in \Hom_\F(\phi(Q),S)$. The proof of Theorem~\ref{KnThm-algebraic} reduces quickly to the case that $\m{G}$ is a subsystem of $\F$. We first make explicit what 
the assumption in Theorem~\ref{KnThm-algebraic} then means, and state
this as a lemma.

\begin{lemma}\label{bijection} 
Let $\F$ be a fusion system on a finite $p$-group $S$ and let $\G$ be
a sub-fusion system of $\F$ on $T \leq S$.
Suppose $Q$ is a
finite $p$--group (not necessarily a subgroup of $S$). The induced map $\Rep(Q,\G) \rightarrow \Rep(Q,\F)$
is surjective if and only if every epimorphic image of $Q$ in
$S$ is $\F$--conjugate to a subgroup of $T$. It is injective, if and
only if $\G$ controls fusion on the epimorphic images of $Q$ in $T$,
i.e., for any epimorphic image $Q'\leq T$ of $Q$ we have $
\Hom_\F(Q',T)=\Hom_\G(Q',T)$.   \qed
\end{lemma}

The next lemma, together with
Theorem~\ref{MainThm}, will easily imply Theorem~\ref{KnThm-algebraic}.
\begin{lemma}\label{sylowiso}
 Let $\F$ be a saturated fusion system on a finite $p$-group $S$ and
 let $\G$ be a saturated subsystem of $\F$ on $T \leq S$. 
Suppose that
 there exists an $\F$--centric subgroup $Q \leq T$ with
 $\Hom_\F(Q,T)=\Hom_\G(Q,T)$. Then $\Aut_\F(T)=\Aut_\G(T)$ and
 $T=S$. 
\end{lemma}

\begin{proof}
As $T$ is a finite $p$--group, there is a finite chain
$$Q = T_0 \vartriangleleft T_1  \vartriangleleft \cdots \vartriangleleft T_n = T$$
with $T_{i+1} = N_T(T_i)$ for $0\leq i<n$. Note that, as $Q$ is $\F$--centric, every $T_i$ is $\F$--centric and thus also $\G$--centric.
We want to show that
$\Aut_\F(T)=\Aut_\G(T)$ by proving that
\begin{equation*}\label{*}\tag{$**$}
\Hom_\F(T_i,T)=\Hom_\G(T_i,T) \mbox{, for all }0\leq i\leq n,
\end{equation*}
by induction on $i$. For $i=0$ the claim is true by assumption. Let now $0\leq i<n$ such
that $\Hom_\F(T_i,T)=\Hom_\G(T_i,T)$. Let $\gamma\in
\Hom_\F(T_{i+1},T)$. Then $\psi=\gamma|_{T_i}\in
\Hom_\F(T_i,T)=\Hom_\G(T_i,T)$. As $T_i$ is $\G$--centric, $\gamma(T_i)$ is fully $\G$--centralized and $C_T(\gamma(T_i))\leq \gamma(T_i)$; so by
Lemma \ref{FSEl1}\eqref{b}, applied to $\gamma$ and $\G$ in the roles
of $\phi$ and $\F$, $\psi$ extends to $\hat{\psi}\in
\Hom_\G(T_{i+1},\gamma(T_{i+1}))$. Then $\hat{\psi}^{-1}\circ\gamma\in
C_{\Aut_\F(T_{i+1})}(T_i)$ and, by \cite[Prop.~A.8]{Broto/Levi/Oliver:2003a} or \cite[Lem.~I.5.6]{Aschbacher/Kessar/Oliver:2011a},
$C_{\Aut_\F(T_{i+1})}(T_i)=\Aut_{Z(T_i)}(T_{i+1})$. 
We conclude that $\hat{\psi}^{-1}\circ\gamma\in
\Aut_{Z(T_i)}(T_{i+1})\leq \Aut_\G(T_{i+1})$ 
and thus $\gamma\in \Hom_\G(T_{i+1},T)$, i.e.,  \eqref{*}
holds. So $$\Aut_\G(T) = \Hom_\G(T,T) = \Hom_\F(T,T) = \Aut_\F(T).$$
 If $\Aut_\F(T)=\Aut_\G(T)$
then, in particular, $\Aut_\F(T)/\Inn(T)$ has order prime to $p$, by
the Sylow axiom for $\G$, 
and so
$\Aut_S(T)=\Inn(T)$. Since $Q\leq T$ is $\F$--centric, this implies that
$N_S(T)=T$, and thus $S=T$. 
\end{proof}

\begin{proof}[{Proof of Theorem~\ref{KnThm-algebraic}}]
We only prove the claim about fusion systems, as the claim about
groups is a special case. First, it is obvious that $T \to
S$ has to be a monomorphism, since if an element is conjugate to the
trivial element, it is trivial. Hence, we may consider $\G$ as a subsystem
of $\F$. Choose a subgroup $A\leq T$ such that $A$ is of maximal order
among the abelian subgroups of $T$. The assumptions, together with Lemma~\ref{bijection},
imply that
every abelian subgroup of $S$ is $\F$--conjugate to a subgroup of $T$,
so $A$ is of maximal order among the abelian subgroups of
$S$, and hence $\F$--centric.  Moreover, again by Lemma~\ref{bijection},  $\Hom_\G(A,T) = \Hom_\F(A,T)$. 
Lemma~\ref{sylowiso}
now shows that $T=S$. This reduces us to a special case of the
setup of Theorem~\ref{MainThm}, and the result follows. 
\end{proof}

\section{Proofs of Theorems \ref{MainCohomThm} and  \ref{KnThm} }\label{topsec}

\begin{proof}[{Proof of Theorem~\ref{MainCohomThm}}]
By Theorem~\ref{MainThm} we just need to verify that an $F$--isomor\-phism on
cohomology rings implies that $H$ controls fusion in $G$ on elementary
abelian $p$--groups. However this is the statement of
\cite[Prop.~10.9(ii)$\Rightarrow$(i)]{Quillen:1971b+c} (see also \cite{Alperin:2006a}).
\end{proof}

Before proving Theorem~\ref{KnThm} we state a lemma
explaining the condition on $n$, whose proof is elementary and seems
best left to the reader. Below $\bZ_p$ denotes the $p$--adic integers.

\begin{lemma} \label{smalllemma} For a homomorphism $\phi\co H \to G$ of
  finite groups, and a fixed natural number $n$, $\Rep(\bZ_p^n,H)
  \xrightarrow{\sim} \Rep(\bZ_p^n,G)$ if and only if
  $\Rep(A,H) \xrightarrow{\sim} \Rep(A,G)$ for all finite abelian
  $p$--groups $A$ with $\prk(A) = n$.
Furthermore, isomorphism for a fixed positive $n
\geq \min\{\prk(G),\prk(H) +1\}$ implies $\prk(G) = \prk(H)$ and
isomorphism for all $n$. \qed
\end{lemma}

In further preparation for the proof of Theorem~\ref{KnThm}, we briefly recall the
HKR character theorem \cite[Thm~C]{Hopkins/Kuhn/Ravenel:2000a}:
For any multiplicative cohomology theory $E$ and finite group $G$, taking $E^*$--cohomology induces a
map 
\[
\Rep(\bZ_p^n,G) \xlongrightarrow{}
\Hom_{E^*\mbox{-}\mathrm{alg}}(E^*(BG),E^*_{\mathrm{cont}}(B\bZ_p^n))
\]
with $E^*_{\mathrm{cont}}(B\bZ_p^n) = \colim_r E^*(B(\bZ/p^r)^n)$, since any
$\alpha\co \bZ_p^n \to G$ factors canonically through $(\bZ/p^r)^n$ for $r$ large.
By adjunction we can view this as an
$E^*$--algebra homomorphism
\begin{equation}\label{charactermap}
E^*(BG) \xlongrightarrow{} \prod_{\Rep(\bZ_p^n,G)} E^*_{\mathrm{cont}}(B\bZ_p^n)
\end{equation}
where the right-hand side is $E^*_{\mathrm{cont}}(B\bZ_p^n)$--valued
functions on the finite set $\Rep(\bZ_p^n,G)$, with point-wise  multiplication.

The map \eqref{charactermap} is the {\em $n$--character map}  and the HKR character theorem
\cite[Thm.~C]{Hopkins/Kuhn/Ravenel:2000a} says that, for certain $E$, this becomes an
isomorphism after suitable localization. 
More precisely, assume that $E =
E(n)$, so $E^*(BS^1) \cong
E^*\llbracket x\rrbracket$, $|x|=2$, and define $L(E^*)$ to be the
ring of fractions of $E^*_{\mathrm{cont}}(B\bZ_p^n)$ obtained by inverting $\alpha^*(x)$ for all
for all non-zero $\alpha \in \Hom_{\mathrm{cont}}(\bZ_p^n, S^1) \cong
(\bZ/p^\infty)^n$. Then, by \cite[Thm.~C]{Hopkins/Kuhn/Ravenel:2000a},
 $L(E^*)$ is faithfully flat over $E(n)^*[\frac1p]$ (and in
particular non-zero) and \eqref{charactermap} induces an isomorphism
\begin{equation}\label{HKR-theorem}
L(E^*) \otimes_{E^*[\frac1p]} E^*(BG)[\frac1p] \xlongrightarrow{\sim}
\prod_{\Rep(\bZ_p^n,G)} L(E^*)
\end{equation}

\begin{proof}[{Proof of Theorem~\ref{KnThm}}]
By the assumption of the theorem and the HKR character isomorphism
\eqref{HKR-theorem} we have an isomorphism
\begin{equation}\label{res} \prod_{\Rep(\bZ_p^n,G)} L(E^*) \xlongrightarrow{\sim}
  \prod_{\Rep(\bZ_p^n,H)} L(E^*)
\end{equation}
given by precomposing with the natural map $\Rep(\bZ_p^n,H) \to
\Rep(\bZ_p^n,G)$. Since $L(E^*) \neq 0$ we conclude that
$\Rep(\bZ_p^n,H) \to \Rep(\bZ_p^n,G)$ is an isomorphism. By
the assumption on $n$ and Lemma~\ref{smalllemma} this implies that
$\Rep(A,H) \to \Rep(A,G)$ is an isomorphism for all finite abelian
groups, and Theorem~\ref{KnThm} now follows from
Theorem~\ref{KnThm-algebraic}. 
\end{proof}

\section{Variations on the results and further comments}

In this final section we elaborate on some supplementary results alluded to in the introduction.

\begin{rk}[A variety version of
  Theorem~\ref{MainCohomThm}] 
\label{fisorem} 
In Theorem~\ref{MainCohomThm} we can replace the assumption of
$F$--isomorphism by the assumption that the map $\iota^* \co H^*(G;\bar \bF_p) \to H^*(H;\bar \bF_p)$  induces a
bijection of maximal ideals, by referencing
\cite[Prop.~10.9(iii)$\Rightarrow$(i)]{Quillen:1971b+c}, and noting that the
maximal ideal spectrum of $H^*(G;\bar \bF_p)$ identifies with $ \Hom_{\bar
  \bF_p\mbox{-}\mathrm{alg}}(H^*(G;\bar \bF_p),\bar \bF_p) =
\Hom_{\mathrm{rings}}(H^*(G;\bF_p),\bar \bF_p)$.

In general a finite morphism $f: A \to B$ of
finitely generated $\bF_p$--algebras is an $F$--isomorphism if and
only if it induces a {\em variety
isomorphism}, i.e., a bijection $\Hom_{\mathrm{rings}}(B,\Omega) \xrightarrow{\sim}
\Hom_{\mathrm{rings}}(A,\Omega)$ for all algebraically closed fields
$\Omega$ \cite[Prop.~B.8-9]{Quillen:1971b+c}. But to get the same fusion on elementary
abelian $p$-subgroups we in fact just need a bijection on $\Hom_{\mathrm{rings}}(-,\Omega)$, for some proper
field extension $\Omega$ of $\bF_p$, by properties of the Quillen
stratification; see \cite[\S9-10]{Quillen:1971b+c} and also \cite[\S9.1]{Evens:1991a}.
\end{rk}

\begin{eg}[An isomorphism
  modulo $\Nil_n$ for $p=2$ which does not control
  $p$--fusion]\label{th:GnHn}
For any $n$, let $G_n = (2A_4)^n$, $P_n = (Q_8)^n$ and $H_n = \ker(\psi)$, where
$\psi\co G_n \to G_n/P_n \cong (C_3)^n \to
C_3$ is given by $(g_1,\dots,g_n)\mapsto g_1\dots g_n$. Note that
$H_n$ does not control $p$--fusion in $G_n$.  
We however claim that the restriction $H^*(G_n;\bF_2)\to H^*(H_n;\bF_2)$ is an
isomorphism modulo  $\Nil_n$, as defined in \cite[Ch.~6]{Schwartz:1994a}, hence showing that Theorem~\ref{MainCohomThm} fails
severely for $p=2$ ($F$--isomorphism is equivalent to isomorphism
modulo $\Nil_1$):
Recall that $H^*(Q_8;\bF_2) \cong
H^{<4}(Q_8;\bF_2) \otimes \bF_2[z]$, with $|z|=4$, where the action of
$2A_4/Q_8 \cong C_3$ on $\bF_2[z]$ is trivial, while on
$H^{<4}(Q_8;\bF_2)$ it is trivial in degrees 0 and 3, and degrees 1 and 2
consists of the two dimensional irreducible 
$\bF_2C_3$--module $V$.
Since $G_n$ and $H_n$ both have Sylow $2$--subgroup $P_n$, the
restriction map  $H^*(G_n;\bF_2)\to H^*(H_n;\bF_2)$ is injective, and
the cokernel is a tensor product of $\bF_2[z_1,\dots,z_n]$  with a certain
finite module $M$, given as the sum of the non-trivial irreducible
$G_n/P_n$--representations on $H^{<4}(Q_8;\bF_2)^{\otimes n}$ which restrict trivially to $H_n/P_n$.
Using the definition of $\Nil_m$ \cite[Ch.~6]{Schwartz:1994a}, the largest $m$ for
which the restriction map is an isomorphism modulo $\Nil_m$
therefore is the first degree where $M$ is non-zero.
To determine this degree we extend coefficients to $\bF_4$ and use Frobenius reciprocity
$\Hom_{\bF_4(H_n/P_n)}(\bF_4,-) \cong
\Hom_{\bF_4(G_n/P_n)}((\bF_4){\ua _{H_n}^{G_n}},-)$, and note that
$ (\bF_4){\ua_{H_n}^{G_n}} \cong (\bF_4\otimes\dots\otimes \bF_4)
\oplus (\omega\otimes\dots\otimes \omega) \oplus (\bar\omega\otimes 
\dots\otimes \bar\omega)$,
for $\bF_4$, $\omega$ and $\bar\omega$ the three 1-dimensional
$\bF_4 C_3$--modules.
In this notation, we have to locate the first copy of 
$\omega\otimes\dots\otimes\omega$ or $\bar\omega\otimes\dots\otimes\bar\omega$
in $(H^{<4}(Q_8;\bF_4))^{\otimes n}$.
This occurs for the first time in degree $n$, where there is a summand $V\otimes\dots\otimes V$,
which over $\bF_4$ is $(\omega \oplus
\bar\omega)\otimes\dots\otimes(\omega\oplus \bar\omega)$ completing
the proof of the claim.
\end{eg}

\begin{rk}[A generalization of Mislin's theorem via Theorem~\ref{KnThm}] \label{rem:highdegcohom2}
A notion of equivalence stronger than $F$--isomorphism, and in fact
also than that of Example~\ref{th:GnHn}, is isomorphism in large degrees.
If a homomorphism $\phi\co H \to G$ induces an isomorphism
in mod $p$ cohomology in large degrees, we can use
Theorem~\ref{KnThm} to see that $|\ker(\phi)|$ and $|G:\phi(H)|$
are coprime to $p$ and that $\phi(H)$ controls $p$--fusion in $G$,
providing a new proof of a strengthening of Mislin's theorem first obtained in  \cite[Cor.~3.4]{Mislin:1993a} (cf.\ also
\cite[Thm.~1.1]{Benson/Carlson/Robinson:1990a}):
By finiteness of group
cohomology the induced map between
$E^2$--terms of $E(n)^*$--Atiyah--Hirzebruch spectral
sequences \cite[Thm.~12.2]{Boardman:1999a} has kernel and cokernel a finite
$p$--group in each total degree. We therefore deduce an
isomorphism $E(n)^*(BG)[\frac1p] \xrightarrow{\sim}
E(n)^*(BH)[\frac1p]$ by spectral sequence comparison
\cite[Thm.~7.2]{Boardman:1999a}, and the claim now follows from Theorem~\ref{KnThm}.

It is perhaps
interesting to note that this proof structurally mirrors Atiyah's 1961
proof of his $p$--nilpotency criterion
  \cite{Atiyah:1961a}\cite[Thm.~1.3]{Quillen:1971a}, which says that a
  Sylow inclusion $S < G$ controls $p$--fusion if it induces an isomorphism in mod $p$ cohomology in
  sufficiently high dimension: Reinterpreting
  \cite[p.~362]{Quillen:1971a}, Atiyah uses his version of the Atiyah--Hirzebruch
  spectral sequence \cite[Thm.~5.1]{Atiyah:1961a} to conclude that $K^*(BG;\bZ_p)[\frac1p] \xrightarrow{\sim}
K^*(BS;\bZ_p)[\frac1p]$. It now follows from the Atiyah--Segal
completion theorem \cite[Thm.~7.2]{Atiyah:1961a} that $S$ and $G$ have the same fusion on cyclic
$p$--subgroups, and hence the same $p$--fusion by \cite[Satz~IV.4.9]{Huppert:1967a}.
\end{rk}

\begin{rk}[A variety version of Theorem~\ref{KnThm}
  and the role of inverting $p$]\label{spec-remark} 
Also in Theorem~\ref{KnThm} it is enough to
assume a variety isomorphism: If  $\phi^*\co  E(n)^*(BG)[\frac 1p] \to
  E(n)^*(BH)[\frac1p]$ induces a
  bijection on $\Hom_{\mathrm{rings}}(-,\Omega)$ for all algebraically closed
  fields $\Omega$, then the same holds after
  extending scalars along  $E(n)^*[\frac1p] \to L(E^*)$. Hence \eqref{HKR-theorem} shows that \eqref{res}
 induces a bijection
$ \coprod_{\Rep(\bZ_p^n,H)} \Hom_{\mathrm{rings}}(L(E^*),\Omega)  \xrightarrow{\sim}
  \coprod_{\Rep(\bZ_p^n,G)} \Hom_{\mathrm{rings}}(L(E^*),\Omega)$
for any algebraically closed field $\Omega$, so $\Rep(\bZ^n_p,H)
\xrightarrow{\sim} \Rep(\bZ_p^n,G)$, and
Theorem~\ref{KnThm} follows from Theorem~\ref{KnThm-algebraic} as above.

In \cite[\S 3]{Greenlees/Strickland:1999a}
Greenlees--Strickland 
explain how the variety of $E(n)^*(BG)[\frac1p]$
constitutes a `zeroth pure stratum' of a chromatic stratification
of the formal spectrum of $E(n)^*(BG)$. Hence
having an isomorphism on  $E(n)^*(-)[\frac1p]$ is a
priori significantly weaker than having isomorphism on
$E(n)^*(-)$ or the formal spectrum ${\operatorname{Spf}}(E(n)^*(-))$.
\end{rk}

\begin{rk}[Theorem~\ref{KnThm} in relationship to other results in stable
  homotopy theory] \label{KnWhitehead} To illuminate the
  assumptions in Theorem~\ref{KnThm} we note
  that a map induces an isomorphism on $E(n)$ (without inverting $p$) if and only if
it induces isomorphism on the corresponding uncompleted Johnson--Wilson
theory, or isomorphism on $K(i)$ for all $i
\leq n$  (see \cite[Thm.~2.1]{Ravenel:1984a} and
\cite[Lec.~23]{Lurie:2010a}). This in turn happens if and only if it
induces isomorphism on just $K(n)$, by a result of Bousfield \cite[Thm.~1.1]{Bousfield:1999a}. 

Homotopy theorists may wonder if there exists a `purely homotopic'
proof of Theorem~\ref{KnThm}. We do not know such a proof, but
combining  Mislin's original theorem \cite{Mislin:1990a} with some
deep results in homotopy theory, one can get a weaker statement that $E(n)^*$--isomorphism for a quite large $n$ (and without inverting $p$) implies that  $H$ controls $p$--fusion in
$G$. We briefly explain this:  Bousfield proved in 1982 a `$K(n)$--Whitehead theorem'
stating that a map between spaces which is an isomorphism on $K(n)^*$
also induces an isomorphism on $H^i(-;\bF_p)$ for $i \leq n$ (see
\cite[Ex.~8.4]{Bousfield:1982a} and
\cite[Thm.~1.4]{Bousfield:1999a}). 
The claim now follows since it is possible to give a
large constant $n$, depending on the Sylow subgroup, such that isomorphism in $H^i(-;\bF_p)$ for $i \leq n$ implies
isomorphism on $H^*(-;\bF_p)$, e.g., using results of Symonds
\cite[Prop.~10.2]{Symonds:2010a} that say that the generators and relations of group cohomology are at most in degree
$2k^2$, where $k$ is the minimal dimension of a 
faithful complex representation of $G$. Observe the bound needs to
depend on more than the $p$--rank: For any $n$ we can pick $p$ such
that $2n \mid p-1$. In this case $\bF_p \rtimes  C_n < \bF_p \rtimes C_{2n}$ induces an
isomorphism on $H^i(-;\bF_p)$ for $i<n$ without controlling
$p$--fusion. This is in
 contrast to the
 Hup\-pert--Thompson--Tate $p$--nilpotency criterion
 \cite{Tate:1964a}, which states that an inclusion of a Sylow $p$--subgroup that induces isomorphism on
 $H^1(-;\bF_p)$ 
 controls $p$--fusion.
\end{rk}

\begin{acknowledgements}
Our interest
in Theorem~\ref{MainCohomThm} was piqued by a discussion during the problem session at
the August 2011 workshop on Homotopical Approaches to Group Actions
in Copenhagen with Peter Symonds and others, and also stimulated by
\cite{arXiv:1107.5158v2}.  
We thank Mike Hopkins for pointers concerning the relationship between Theorem~\ref{KnThm}
and classical stable homotopy theory, explained in
Remark~\ref{KnWhitehead}, and Lucho  Avramov, Nick Kuhn, and Neil
Strickland for other literature references. 
\end{acknowledgements}

\bibliographystyle{amsplain}
\bibliography{repcoh}

\end{document}